\documentclass[12pt,a4paper,reqno]{amsart}
\usepackage{graphics,epic}
\usepackage{amsmath,amssymb, amsthm}
\usepackage[all,2cell]{xy}
\usepackage{color}

\newtheorem{thm}{Theorem}[section]
\newtheorem*{theorem*}{Theorem}

\newtheorem{lem}[thm]{Lemma}
\newtheorem{prop}[thm]{Proposition}

\newtheorem*{conjecture*}{Conjecture}

\newcommand\mtx[1]{\begin{bmatrix} #1 \end{bmatrix}}

\newcommand{\ch}{{\mathcal H}}

\newcommand{\E}{{\mathbb E}}

\newcommand{\Z}{\mathbb{Z}}

\newcommand{\C}{\mathbb{C}}

\def\wt{{\rm wt}}

\def\C{{\mathbb C}}

\def\R{{\mathbb R}}

\def\Z{{\mathbb Z}}

\def\H{{\mathbb H}}

\def\1{{\bf 1}}
\def\tr{{\rm tr}}

\def \a{\alpha}

\def \e{\epsilon}
\def \h{\mathfrak{h}}
\def \l{\lambda}

\def \g{\mathfrak{g}}
\def\wh{\widehat{\mathfrak h}}
\def\wg{\widehat{\mathfrak g}}

\def\<{\langle}
\def\>{\rangle}

\def\ch{{\rm ch}}
\def\E{{\mathbb E}}

\begin{document}
\title[Parafermion vertex operator algebras]{Trace functions of  the parafermion vertex operator algebras}
\author{Chongying Dong}
\address{ Department of Mathematics, University of California, Santa Cruz, CA 95064}
 \email{dong@ucsc.edu}
\thanks{The first  author was supported by China NSF grant 11871351}
\author{Victor Kac}
\address{Department of Mathematics, MIT, 77 Mass. Ave, Cambridge, MA 02139}
\email{kac@math.mit.edu}
\thanks{}
\author{Li Ren}
\address{School of Mathematics, Sichuan University,
 Chengdu 610064 China }
 \email{renl@scu.edu.cn}
 \thanks{The third author was supported by  China NSF grant 11671277}
\maketitle

\begin{abstract}
The trace functions  for the Parafermion vertex operator algebra associated to any finite dimensional simple Lie algebra $\g$ and any positive integer $k$ are studied and an explicit modular transformation formula of the trace functions is obtained.
\end{abstract}

\section{Introduction}

This paper is a continuation of our study of the Parafermion vertex operator algebra $K(\g,k)$ associated to any finite dimensional simple Lie algebra $\g$ and positive integer $k.$ In particular, we determine
the modular transformation formula
of the trace functions for the Parafermion vertex operator algebras.

The Parafermion vertex operator algebra $K(\g,k)$ is the commutant of the Heisenberg vertex operator algebra $M_{\wh}(k)$ in
the affine vertex operator algebra $L_{\wg}(k,0),$ and can also be regarded as the commutant of the lattice vertex operator
algebra $V_{\sqrt{k}Q_L}$ in   $L_{\wg}(k,0)$ where $Q_L$ is the lattice spanned by the long roots of $\g.$
While the $C_2$-cofiniteness  of the Parafermion vertex operator algebras was proved in \cite{DW2} and \cite{ALY1}, the rationality was established in \cite{DR} with the help of a result in \cite{CM} on the abelian orbifolds for rational and $C_2$-cofinite vertex operator algebras.  The irreducible modules of $K(\g,k)$ were classified in \cite{ALY1} for $\g=sl_2$ and
in \cite{DR}, \cite{ADJR} for general $\g,$ and the fusion rules  were computed in \cite{DW4} for $sl_2$ and in \cite{ADJR} for general $\g$ with the help of quantum dimensions.
See recent work in \cite{DLY}, \cite{DLWY}, \cite{DW1}, \cite{DW4}, \cite{ALY2} and \cite{JW}
for other topics on Parafermion vertex operator algebras.

The trace functions $Z_M(v,\tau)$  are the main objects in this paper where $M$ is an irreducible $K(\g,k)$-module
and $v\in K(\g,k).$ Since $K(\g,k)$ is rational and $C_2$-cofinite, the space spanned by $Z_M(v,\tau)$ for the irreducible modules $M$ affords a representation of the modular group $SL_2(\Z)$ \cite{Z}.
Our goal is to determine this representation
explicitly. The irreducible $K(\g,k)$-modules are labeled by $M^{\Lambda,\lambda}.$ Here  $\Lambda$ is a dominant weight
of $\g$ such that $\<\Lambda,\theta\>\leq k,$ $\<,\>$ is the normalized invariant bilinear form on $\g$ so that the squared length of a  long root is $2,$ $\theta$ is the maximal root,  $\lambda\in \Lambda+Q $
modulo $ kQ_L$ and $Q$ is the root lattice.
In the case $v=\1$ the functions $\chi_{M^{\Lambda,\lambda}}(\tau)$ are a special kind of branching functions studied previously in \cite{K}, \cite{KP}. Moreover,  $\chi_{M^{\Lambda,\lambda}}(\tau)/\eta(\tau)^l$ is
the string function in \cite{K},  \cite{KP} where $l$ is
the rank of $\g.$ In fact, explicit modular transformation formulas for  branching functions were obtained in \cite{K}. Our main result in this paper is that the same transformation formula for the branching functions is valid
for the trace function with $\1$ replaced by any $v\in K(\g,k).$

The main idea is to use the modular transformation formulas for the affine vertex operator algebra   $L_{\wg}(k,0)$ and
lattice vertex operator algebras $V_{\sqrt{k}Q_L}.$ There are two modular transformation formulas for the affine vertex operator algebra   $L_{\wg}(k,0)$ given in \cite{K} explicitly  and \cite{Z} abstractly.  Using a transformation formula for the abstract theta functions studied in \cite{Kr},  \cite{M}, one can easily show that these two transformation formulas are the same. Unfortunately, this can only give the modular transformation formula for the character  $\chi_{M^{\Lambda,\lambda}}(\tau),$ not for the one point function
$Z_{M^{\Lambda,\lambda}}(v,\tau).$ Again,results in \cite{Kr} on abstract theta functions including some vector $w$ help us to solve the problem.  To explain what kind of vector $w$ is, we recall from \cite{DM2} that if $V=\oplus_{n\geq 0}V_n$ is a rational , $C_2$-cofinite, simple
vertex operator algebra of CFT type, then $V_1$ is a reductive Lie algebra whose rank (the dimension of a Cartan subalgebra)
is less than or equal to the central charge. Fix a Cartan subalgebra $\h$ of $V_1,$ vector $w$ satisfies conditions
$h_nw=0$ for all $h\in \h$ and $n\geq 0.$ In the case $V= L_{\wg}(k,0),$ the condition is exactly equivalent to $w\in K(\g,k).$
The rest of the proof is similar to that given in \cite{K} for the branching functions.

The paper is organized as follows: We review basics on vertex operator algebras and their modules in Section 2 including
rationality, $C_2$-cofiniteness and modular transformation results on the trace functions associated to a rational and $C_2$-cofinite vertex operator algebra $V.$
Following \cite{M} and \cite{Kr}, we discuss Section 3 the modular transformation formula for  generalized theta functions $\chi_i(w,v,q)=\tr_{M^i}o(w)e^{2\pi
io(v)}q^{L(0)-c/24}$ where $M^i$ is an irreducible $V$-module and $v\in \h$ and $w$ is defined as before.   In Section 4 we discuss the affine vertex operator algebras $L_{\wg}(k,0)$  and the modular  transformation formula of the specialized characters of  the level $k$   integrable highest weight modules for the the affine Kac-Moody algebra $\wg.$ We recall the various known results on the Parafermion vertex operator algebra $K(\g,k)$ such as rationality and classification of irreducible modules. Section 6 deals with
the modular transformation formula for the 1 point function $Z_{M^{\Lambda,\lambda}}(v,\tau).$ The $T$-matrix is easy and the most effort is on the $S$-matrix.  The idea and method for finding the $S$-matrix is similar to that in \cite{K} and \cite{KP}.

\section{Basics}
\setcounter{equation}{0}
In this section we review the basic  on vertex operator algebra. Let $V=(V,Y,\1,\omega)$ be a vertex operator algebra (cf. \cite{B} and \cite{FLM}).

A vertex operator algebra  $V$ is of {\em CFT type} if $V$ is simple,  with respect to $L(0)$ one has $V=\oplus_{n\geq 0}V_n$ and $V_0=\C {\bf 1}$ \cite{DLMM}.

A vertex operator algebra $V$ is called $C_2$-{\em cofinite} if $\dim V/C_2(V)<\infty$ where
$C_2(V)$ is the subspace of $V$ spanned by $u_{-2}v$ for $u,v\in V$ \cite{Z}.

A {\em weak} $V$-module  $M=(M,Y_M)$ is a vector space equipped with a linear map
$$\begin{array}{ll}
Y_M: & V\to (\mbox{End}\,M)[[z^{-1},z]]\label{map}\\
& v\mapsto\displaystyle{Y_M(v,z)=\sum_{n\in \Z}v_nz^{-n-1}\ \ (v_n\in\mbox{End}\,M)}
\mbox{ for }v\in V\label{1/2}
\end{array}$$
satisfying the following conditions for $u,v\in V$,
$w\in M$:
\begin{eqnarray*}\label{e2.1}
& &v_nw=0\ \ \  				
\mbox{for}\ \ \ n\in \Z \ \ \mbox{sufficiently\ large};\label{vlw0}\\
& &Y_M({\bf 1},z)=1;\label{vacuum}
\end{eqnarray*}
 \begin{equation*}\label{jacobi}
\begin{array}{c}
\displaystyle{z^{-1}_0\delta\left(\frac{z_1-z_2}{z_0}\right)
Y_M(u,z_1)Y_M(v,z_2)-z^{-1}_0\delta\left(\frac{z_2-z_1}{-z_0}\right)
Y_M(v,z_2)Y_M(u,z_1)}\\
\displaystyle{=z_2^{-1}\delta\left(\frac{z_1-z_0}{z_2}\right)
Y_M(Y(u,z_0)v,z_2)}.
\end{array}
\end{equation*}

An ({\em ordinary}) $V$-module is a  weak $V$-module $M$ which
is $\C$-graded
$$M=\bigoplus_{\lambda \in{\C}}M_{\lambda} $$
such that $\dim M_{\l}$ is finite and $M_{\l+n}=0$
for fixed $\l$ and $n\in {\Z}$ small enough, where
$M_{\l}$ is the eigenspace for $L(0)$ with eigenvalue $\lambda:$
$$L(0)w=\l w, \ \ \ w\in M_{\l}.$$
Let $M$ be an ordinary $V$-module.  We denote $v_{n-1}$ by $o(v)$ for $v\in V_n$ and  extend to $V$ linearly.  Then $o(v)M_{\lambda}\subset M_{\lambda}$ for all $\lambda\in\C.$

An {\em admissible} $V$-module is
a  weak $V$-module $M$ which  carries a
${\Z}_{+}$-grading
$$M=\bigoplus_{n\in {\Z}_{+}}M(n)$$
($\Z_+$ is the set all nonnegative integers) such that if $r, m\in {\Z} ,n\in {\Z}_{+}$ and $a\in V_{r}$
then
$$a_{m}M(n)\subseteq M(r+n-m-1).$$
Note that any ordinary module is an admissible module.

A vertex operator algebra $V$ is called {\em rational} if any admissible
module is a direct sum of irreducible admissible modules \cite{DLM1}.

The following result is proved  in \cite{DLM2}  and \cite{DLM3}.
\begin{thm}\label{DLMrational}  Assume that  $V$ is rational.

(1)  There are only
finitely many inequivalent irreducible admissible modules $V=M^0,...,M^p$
and each irreducible admissible module is an ordinary module. Each $M^i$ has weight space decomposition
$$M^i=\bigoplus_{n\geq 0}M^i_{\lambda_i+n}$$
where $\lambda_i\in\C$ is a complex number such that $M^i_{\lambda_i}\ne 0$ and $M^i_{\lambda_i+n}$ is the eigenspace of $L(0)$ with eigenvalue $\lambda_i+n.$ The $\lambda_i$ is called the weight of $M^i.$

(2) If $V$ is both rational and
$C_2$-cofinite, then $\lambda_i$ and central charge $c$ are rational numbers.
\end{thm}

In the rest of this  paper we assume that $V$ is a strong rational vertex operator algebra. That is $V$ satisfies  the following:
\begin{enumerate}
\item[(V1)] $V=\oplus_{n\geq 0}V_n$ is a simple vertex operator algebra  of CFT type,
\item[(V2)] $V$ is $C_{2}$-cofinite and  rational,
\item[(V3)] The conformal weights $\lambda_i$ are nonnegative and $\lambda_i=0$ if and only if $i=0.$
\end{enumerate}

Using the assumption (V3) we know that $V$ and its contragredient module $V'$ \cite{FHL} are isomorphic $V$-modules.
From \cite{DM2} we know $V_1$ is a reductive Lie algebra with $[u,v]=u_0v$ for $u,v\in V_1.$ Moreover, the rank $V_1$ is less than or equal to the central charge $c$ of $V.$ For any $u\in V_1,$ $u_0$ is a derivation of $V$ in the sense that
$[u_0,Y(v,z)]=Y(u_0v,z)$ for any $v\in V$ and $u_0\omega=0.$ Furthermore, $e^{u_0}$ is an automorphism of $V.$

Recall from \cite{FHL}
that a bilinear from $(\cdot, \cdot)$ on $V$ is called {\rm invariant} if
\begin{eqnarray*}
(Y(u,z)v,w)=(u, Y(e^{zL(1)}(-z^{-2})^{L(0)}v,z^{-1})w)
\end{eqnarray*}
for $u,v,w\in V.$ From our assumptions  there is unique nondegenerate
invariant bilinear form on $V$ \cite{L1}. We shall fix a bilinear
form $(\cdot,\cdot)$ on $V$ so that $(u,v)=u_1v$ for $u,v\in V_1.$
 It is clear from the definition that
$(gu,gv)=(u,v)$ for any automorphism $g$ and $u,v\in V_1.$

The modular transformations of trace functions of irreducible modules of vertex operator algebras \cite{Z} are of main importance  in this paper. Another vertex operator algebra  structure
$(V,Y[\cdot,z],\1,\omega-c/24)$ is defined on $V$ in \cite{Z} with grading
 $$V=\bigoplus_{n\geq 0}V_{[n]}.$$ For $v\in  V_{[n]}$ we write $\wt [v]=n.$
 For $i=0,...,p$ we set
 $$Z_i(v,q)=\tr_{M^i}o(v)q^{L(0)-c/24}$$
 which is a formal power series in variable $q.$
The $Z_i(\1,q)$ which is also denoted by  $\ch_qM^i$ is called the $q$-character of $M^i.$

The  series $Z_i(v,q)$ converges to a holomorphic function $Z_i(v,\tau)$
on $\H=\{\tau\in \C \mid {\rm Im}\,\tau>0\}$ where
 $q=e^{2\pi i\tau}$ with $\tau\in \H$ \cite{Z}.

Recall the modular group $SL_2(\Z)$ and let $\gamma=\left(\begin{array}{cc} a & b\\ c & d\end{array}\right)\in SL_2(\Z).$
\begin{thm}\label{2.3} Let $V$ be a strong rational  vertex operator algebra.

(1) There is  a group homomorphism $\rho: SL_2(\Z)\to GL_{p+1}(\C)$ with $\rho(\gamma)=(\gamma_{ij})$
such that for any $0\leq i\leq p$ and $v\in V,$
\begin{eqnarray*}
Z_i(v, \frac{a\tau+b}{c\tau+d})=({c\tau+d})^{\wt[v]}\sum_{j=0}^p\gamma_{ij}Z_j(v,\tau).
\end{eqnarray*}

(2) Each $Z_i(v,\tau)$ is a modular form of weight $\wt[v]$ over a congruence subgroup.
 \end{thm}

Part (1) of the Theorem was obtained in \cite{Z} and Part (2) was established in \cite{DLN}.
The matrices
$$S=\rho_V\left(\mtx{0 & -1\\ 1 & 0}\right) \text{ and }T=\rho_V\left(\mtx{1 & 1\\ 0 & 1}\right).
$$
are respectively called the \emph{genus one} $S$- and $T$-matrices of $V$.

\section{Modular invariance of the generalized theta functions}
\setcounter{equation}{0}

We review the  modular transformation formula of the theta functions defined on vertex operator algebra given in \cite{M} and \cite{Kr}
for studying the modular invariance of  trace functions for the parafermion vertex operator algebras.

Recall that $V_1$ is a reductive Lie algebra. We fix a Cartan subalgebra $\h$ of $V_1.$ Then the abelian Lie algebra
$\h$ acts on $M^i$ semsimply for all $i.$ Following \cite{M}, we define the generalized theta functions as
$$T_i(v,u,q)=\tr_{M^i}e^{2\pi
i(v_0+(u,v)/2)}q^{L(0)+u_0+(u,u)/2-c/24}$$
for $u,v\in \h.$ The bilinear form on $V_1$ used
in \cite{M} is the negative  of the bilinear form used in this paper.
So our $T_i(v,u,q)$ defined here is exactly the $Z_{M^i}(v,u,q)$ in  \cite{M}. Based on Theorem
\ref{2.3}, a modular transformation law was obtained in \cite{M} with the convergence of $Z_i(v,u,q)$ proved in \cite{DLiuM}.

\begin{thm}\label{m} Let $u,v\in \h$ and
$\gamma=\left(\begin{array}{cc}
a & b\\
c & d
\end{array}\right)\in SL_2(\Z).$ Then $T_i(v,u,q)$
converges to a holomorphic function in the upper half plane with $q=e^{2\pi i\tau}$ and
$$
T_i(v,u,\frac{a\tau+b}{c\tau+d})=\sum_{j=0}^p{\gamma}_{ij}
T_j(dv+bu,cv+au,\tau)$$
where $\gamma_{ij}$ is the same as in Theorem \ref{2.3}.
\end{thm}

Set $\chi_i(v,\tau)=T_i(v,0, \tau).$ Using Theorem \ref{m}  one can  easily   show the following result \cite{DLiuM}.
\begin{prop} \label{p1} Assume that $V$ is a strong rational vertex operator algebra.
Then for $\gamma=\left (\begin{array}{l} a \quad b\\
c \quad d
\end{array}  \right ) \in SL_2(\Z),$ $v\in \h$
$$\chi_s({v \over c \tau + d}, { a \tau + b \over c \tau +d}) =
e^{\pi i ( c (v,v)/ (c \tau +
d))}\sum_{j=0}^p\gamma_{sj}\chi_j(v,\tau).$$
\end{prop}

 As usual we set $o(w)=w_{\wt w-1}$ for homogeneous $w\in V$ and extend to whole $V.$ Now take $w\in V$ such that $h_nw=0$ for all $h\in \h$ and $n\geq 0.$ Also define
$$\chi_i(w,v,q)=\tr_{M^i}o(w)e^{2\pi
io(v)}q^{L(0)-c/24}.$$
The following result \cite{Kr} generalizes  Proposition  \ref{p1}.
\begin{thm}\label{gm}
Let $v\in \h.$ We assume that $\chi_i(w, v,q)$ converges to a holomorphic function  $\chi_i(w, v,\tau)$
in $\H$ with $q=e^{2\pi i\tau}$ for any $v\in\h.$  Then  for
 $\gamma=\left(\begin{array}{cc}
a & b\\
c & d
\end{array}\right)\in SL_2(\Z),$
$$
\chi_i(w, \frac{v}{c\tau+d},\frac{a\tau+b}{c\tau+d})=(c\tau+d)^{\wt[v]}e^{\frac{\pi i  c (v,v)}{ c \tau +
d}}\sum_{j=0}^p{\gamma}_{ij}
\chi_j(w, v,\tau).$$
\end{thm}

\section{Affine  vertex operator algebras}
\setcounter{equation}{0}

In this section we discuss the affine vertex operator algebra  $L_{\wg}(k,0)$  associated to the level $k$ integrable highest weight module for affine Kac-Moody algebra $\wg$ and its irreducible modules.

Let $\g$ be a finite dimensional simple Lie algebra  with a Cartan
subalgebra $\h.$ We denote the corresponding root system by $\Delta$
and the root lattice by $Q.$ Fix an invariant symmetric
nondegenerate bilinear form $\< ,\>$  on $\g$ such that $\<\a,\a\>=2$ if
$\alpha$ is a long root, where we have identified $\h$ with $\h^*$
via $\<,\>.$ We denote the image of $\alpha\in
\h^*$ in $\h$ by $t_\alpha.$ That is, $\alpha(h)=\<t_\alpha,h\>$
for any $h\in\h.$ Fix simple roots $\{\alpha_1,...,\alpha_l\}$
and let $\Delta_+$ be the set of corresponding positive roots.
Denote the highest root by $\theta,$  and the Weyl group by $W.$  Also let $\rho$ be the half sum of positive roots.

Recall that the weight lattice $P$ of $\g$ consists of $\lambda\in \h^*$ such that
$\frac{2\<\lambda,\alpha\>}{\<\alpha,\alpha\>}\in\Z$ for all $\alpha\in \Delta.$ It is well-known that
$P=\bigoplus_{i=1}^l\Z\Lambda_i$ where $\Lambda_i$ are the fundamental weights defined by the equation
$\frac{2\<\Lambda_i,\alpha_j\>}{\<\alpha_j,\alpha_j\>}=\delta_{i,j}$ for $1\leq i,j\leq l.$ Let $P_+$ be the subset
of $P$ consisting of the dominant weight $\Lambda\in P$ in the sense that $\frac{2\<\Lambda,\alpha_j\>}{\<\alpha_j,\alpha_j\>}$ is nonnegative for all $j.$ For any nonnegative
integer $k$ we also let $P_+^k$ be the subset of $P_+$ consisting of $\Lambda$ satisfying $\<\Lambda,\theta\>
\leq k.$

Let $Q=\sum_{i=1}^l\Z\alpha_i$ be the root lattice and $Q_L$ be the sublattice of $Q$ spanned by the long roots. Recall that the dual lattice $Q_L^{\circ}$ consists $\lambda\in \h^*$ such that $\<\lambda,\alpha\>\in\Z$ for
all $\alpha\in Q_L.$ Then $P$ is the dual lattice of $Q_L$ \cite{ADJR}.

Let $\wg=\g\otimes \C[t,t^{-1}]\oplus \C K$ be the affine Lie algebra. Fix a nonnegative integer $k.$ For any $\Lambda\in P_+^k$ let $L(\Lambda)$ be the irreducible highest weight $\g$-module with highest weight $\Lambda$
and $L_{\wg}(k,\Lambda)$ be the unique irreducible $\wg$-module such that $L_{\wg}(k,\Lambda)$ is generated by $L(\Lambda)$ and $\g\otimes t^n L(\Lambda)=0$ for $t>0$ and $K$ acts as constant $k.$ The following result is well known
(cf. \cite{FZ},\cite{LL}):
\begin{thm}\label{affinevoa} The $L_{\wg}(k,0)$  is a simple, rational and $C_2$-cofinite  vertex operator algebra with central charge
$\frac{k\dim{\frak g}}{k+h^{\vee}}$ and  whose irreducible modules
are $L_{\wg}(k,\Lambda)$ for $\Lambda\in P_+^k$ and the weight $n_{\Lambda}$  of $L_{\wg}(k,\Lambda)$ is
$\frac{\<\Lambda+2\rho, \Lambda\>}{2(k+h^{\vee})}$ where  $\rho=\sum_{i=1}^l\Lambda_i$ and
$h^{\vee}$ is the dual Coxeter number.
\end{thm}

Note that $L_{\wg}(k,\Lambda)$ has a decomposition with respect to the action of $\h:$
\begin{equation}\label{e4.1}
L_{\wg}(k,\Lambda)=\oplus_{\lambda\in\Lambda+Q}L_{\wg}(k,\Lambda)(\lambda)
\end{equation}
where $h(0)$ acts on $L_{\wg}(k,\Lambda)(\lambda)$ as constant $\lambda(h)=\<\lambda,h\>$ for $h\in \h.$
Following \cite{K}, define the character
$$\chi_{\Lambda}(h,\tau)=\tr_{L_{\wg}(k,\Lambda)}e^{2\pi i h(0)}q^{L(0)-c/24}$$
for $h\in \h.$
Note that the character defined in \cite{K} has an extra factor $e^{2\pi i ku}$  with $u$ being a complex number. This extra factor
makes the modular transformation formula more beautiful. But from the point of view of vertex operator algebra, this extra factor is not necessary.   From \cite{K} we have
\begin{thm}\label{k} Let $\g$ and $k$ be as before.

(1) $ \{\chi_{\Lambda}(h,\tau)|\Lambda\in P_+^k\}$ are linearly independent functions on $\h\times \H.$

(2) For $\Lambda,\Lambda'\in P_+^k,$  set
 $$S_{\Lambda,\Lambda'}=i^{|\Delta_+|}|P/(k+h^{\vee})Q_L|^{-1/2}\sum_{w\in W}(-1)^{l(w)}e^{-\frac{2\pi i}{k+h^{\vee}}\<w(\Lambda+\rho), \Lambda'+\rho\>}$$
where  $l(w)$ is the length of $w.$  Then
$$\chi_{\Lambda}(\frac{h}{\tau},\frac{-1}{\tau})=e^{\pi ik\<h,h\>/\tau} \sum_{\Lambda'\in P_+^k}S_{\Lambda,\Lambda'}\chi_{\Lambda'}(h,\tau).$$
\end{thm}

Recall Theorem \ref{2.3} and Proposition \ref{p1}.
\begin{lem}\label{same} The $S$-matrices in Theorem \ref{2.3} and Theorem \ref{k} are the same.
\end{lem}
\begin{proof}
Let $V=L_{\wg}(k,0).$ Then $V$ satisfies the assumptions in Theorem \ref{2.3}. Recall that we have defined a bilinear form on $V_1$ such that $(u,v)=u_1v.$ Let $u=a(-1)\1$ and $v=b(-1)\1$
for $a,b\in\g.$ Then $(u,v)=a(1)b(-1)\1=k\<a,b\>.$ Identify Lie algebra $V_1$ with $\g$ we see that $(a,b)=k\<a,b\>.$
We denote the $S$-matrix in Theorem \ref{2.3} by $(S_{L_{\wg}(k,\Lambda), L_{\wg}(k,\Lambda')}).$ Then from Proposition \ref{p1} we  know
$$\chi_{\Lambda}(\frac{h}{\tau},\frac{-1}{\tau})=e^{\pi ik\<h,h\>/\tau} \sum_{\Lambda'\in P_+^k}S_{L_{\wg}(k,\Lambda), L_{\wg}(k,\Lambda')}\chi_{\Lambda'}(h,\tau).$$
The result now is an immediate consequence of Theorem \ref{k} (1).
\end{proof}

\section{Parafermion vertex operator algebras}
\setcounter{equation}{0}

In this section we recall the definition of a parafermion  vertex operator algebra $K(\g,k)$ associated to any finite dimensional simple Lie algebra $\g$ and a positive integer $k.$ We also  discuss some known results on $K(\g,k)$ from \cite{DR}.

Let $\lambda_i\in P$ such that $\lambda_i=\frac{\<\theta,\theta\>}{\<\alpha_i,\alpha_i\>}\Lambda_i$ for $i=1,...,l.$ Then
$\<\alpha_i,\lambda_j\>=\delta_{i,j}$ for all $i,j$ and $Q^{\circ}=\bigoplus_{i=1}^l\Z\lambda_i.$ The following result is immediate
from the relation between $\Lambda_i$ and $\lambda_i.$
\begin{lem}\label{qdual} $P/Q^{\circ}$ is a group of order
$$|P/Q^{\circ}|=\left\{\begin{array} {ll}
1 &  A_l,D_l, E_6,E_7,E_8\\
2 & B_l\\
2^{l-1} & C_l\\
2^2 &  F_4\\
3 & G_2
\end{array}\right.$$
\end{lem}

\begin{lem}\label{QLP} For any simple Lie algebra $\g$ and any positive integer  $k,$ $\frac{1}{k}P/Q^{\circ}$ and
 $Q/kQ_L$ are dual groups, so  that for any  $\beta \in Q,$  $g_{\beta }(\alpha)=e^{2\pi i \<\beta,\alpha\>}$  defines
an irreducible character for  $\frac{1}{k}P/Q^{\circ}$ In particular,  $\frac{1}{k}P/Q^{\circ}$ and  $Q/kQ_L$  are isomorphic groups.
\end{lem}
\begin{proof} Clearly, $g_{\beta}$ defines a irreducible character of $\frac{1}{k}P/Q^{\circ}$ as $Q_L^{\circ}=P.$  Also, $g_{\beta_1}=g_{\beta_2}$ for
if and only if $\beta_1-\beta_2\in kQ_L.$  So   $Q/kQ_L$ is a subgroup of the dual group of $\frac{1}{k}P/Q^{\circ}.$  To finish the proof,
it is enough to show that $|P/Q^{\circ}|=|Q/Q_L|.$ This is obvious if $\g$ is a Lie algebra of type $A, D,E.$
If  $\g$ is a Lie algebra of  other type, we verify the result  case by case using the root systems given in \cite{H}.
 We have already known $|P/Q^{\circ}|$ from Lemma \ref{qdual}. So we only need to compute $|Q/Q_L|.$

 (1)  Type $B_l.$ Let $\E=\R^l$ with the standard orthonormal basis $\{\epsilon_1,...,\epsilon_l\}.$ Then
 $$\Delta=\{\pm\epsilon_i, \pm(\epsilon_i\pm\epsilon_j)|i\ne j\}.$$
 Then $Q=\sum_i^l\Z\epsilon_i$ and $Q_L=\sum_{i\ne j}(\Z(\epsilon_i+\epsilon_j)+\Z(\epsilon_i-\epsilon_j)).$
It is evident that $2Q\subset Q_L$ and  $|Q/Q_L|=2$ with coset representatives $0$ and $\epsilon_1.$

(2) Type $C_l.$ In this case, $$\Delta=\{\pm\sqrt{2}\epsilon_i, \pm\frac{1}{\sqrt{2}}(\epsilon_i\pm\epsilon_j)|i\ne j\}.$$
Then $Q=\frac{1}{\sqrt{2}}\sum_{i\ne j}(\Z(\epsilon_i+\epsilon_j)+\Z(\epsilon_i-\epsilon_j))$ and  $Q_L=\sqrt{2}\sum_{i=1}^l\Z\epsilon_i.$ Thus $|Q/Q_L|=2^{l-1}$ with coset representatives
$a_1\alpha_1+a_2\alpha_2+\cdots +a_{l-1}\alpha_{l-1}$ for $a_i=0,1$ and $\alpha_i=\frac{1}{\sqrt{2}}(\epsilon_i-\epsilon_{i+1}).$

 (3) Type $F_4.$ Let $\E=\R^4.$ Then
 $$\Delta=\{\pm\epsilon_i, \pm(\epsilon_i\pm\epsilon_j), \pm\frac{1}{2}(\epsilon_1\pm \epsilon_2\pm \epsilon_3\pm \epsilon_4)|i\ne j\}.$$
Then $|Q/Q_L|=2^2$ with coset representatives $a\epsilon_3+b\frac{1}{2}(\epsilon_1- \epsilon_2- \epsilon_3-\epsilon_4)$ for   $a,b=0,1.$

 (4) Type $G_2.$ Let $\E$ be the subspace of $\R^3$ orthogonal to $\epsilon_1+\epsilon_2+\epsilon_3.$
 Then
 $$\Delta=\pm\frac{1}{\sqrt{3}}\{\epsilon_i-\epsilon_j, 2\epsilon_1-\epsilon_2-\epsilon_3, 2\epsilon_2-\epsilon_1-\epsilon_3,
  2\epsilon_3-\epsilon_1-\epsilon_2|i\ne j\}.$$
Then $|Q/Q_L|=3$ with coset representatives $a\frac{1}{\sqrt{3}}( \epsilon_1-\epsilon_2)$ for  $a=0,1,2.$ The proof is complete.
\end{proof}

Let $M_{\widehat{\h}}(k)$ be the vertex operator subalgebra of $L_{\widehat{\g}}(k,0)$
generated by $h(-1)\1$ for $h\in \mathfrak h.$
For $\lambda\in
{\mathfrak h}^*,$ denote by  $M_{\widehat{\h}}(k,\lambda)$ the irreducible
highest weight module for $\wh$ with a highest weight vector
$e^\lambda$ such that $h(0)e^\lambda = \lambda(h) e^\lambda$ for
$h\in \mathfrak h.$ The parafermion vertex operator algebra $K(\g,k)$ is the commutant \cite{FZ} of
$M_{\widehat{\h}}(k)$ in $L_{\wg}(k,0).$ We have the following decomposition
\begin{equation}\label{e5.1}
L_{\wg}(k,\Lambda)=\bigoplus_{\lambda\in Q+\Lambda}M_{\widehat{\h}}(k,\lambda)\otimes M^{\Lambda,\lambda}=\bigoplus_{\alpha\in Q}M_{\widehat{\h}}(k,\Lambda+\alpha)\otimes M^{\Lambda,\Lambda+\alpha}
\end{equation}
as $M_{\widehat{\h}}(k)\otimes K(\g,k)$-module. Moreover, $M^{0,0}=K(\g,k)$ and $M^{\Lambda,\lambda}$ is an
irreducible $K(\g,k)$-module \cite{DR}.  Recall equation (\ref{e5.1}). It is easy to see that
\begin{equation}\label{e5.2}
L_{\wg}(k,\Lambda)(\lambda)=M_{\widehat{\h}}(k,\lambda)\otimes M^{\Lambda,\lambda}.
\end{equation}

It is proved in \cite{DW3} that the lattice vertex operator algebra $V_{\sqrt{k}Q_L}$ is a vertex operator subalgebra of $L_{\widehat{\g}}(k,0)$ and the parafermion vertex operator algebra $K(\g,k)$ is also a commutant
of $V_{\sqrt{k}Q_L}$ in $L_{\widehat{\g}}(k,0).$ This gives us another decomposition
\begin{equation}\label{e5.3}
L_{\widehat{\g}}(k,\Lambda)=\bigoplus_{i\in Q/kQ_L}V_{\sqrt{k}Q_L+\frac{1}{\sqrt{k}}(\Lambda+\beta_i)}\otimes M^{\Lambda,\Lambda+\beta_i}
\end{equation}
as modules for $V_{\sqrt{k}Q_L}\otimes K(\g,k),$ where $M^{\Lambda,\lambda}$ is as before,
$Q=\cup_{i\in Q/kQ_L} (kQ_L+\beta_i),$ and $\beta_i\in Q$ is a  representative of $i.$ Moreover, for any $i\in Q/kQ_L,$ we have
$$V_{\sqrt{k}Q_L+\frac{1}{\sqrt{k}}(\Lambda+\beta_i)}\otimes M^{\Lambda,\Lambda+\beta_i}=\bigoplus_{\alpha\in kQ_L}M_{\widehat{\h}}(k,\Lambda+\alpha+\beta_i)\otimes M^{\Lambda,\Lambda+\alpha+\beta_i}$$
where we have used the isomorphism between $K(\g,k)$-modules $M^{\Lambda,\Lambda+\alpha+\beta_i}$ and
 $M^{\Lambda,\Lambda+\beta_i}$ for any $\alpha\in kQ_L$ \cite{DR}. So for any $h\in \h,$ $h(0)\in \wg$ acts  as $\sqrt{k}h(0)$ or $\<h,k\alpha+\Lambda+\beta_i\>$  on $e^{\sqrt{k}\alpha+\frac{1}{\sqrt{k}}(\Lambda+\beta_i)}.$

Let $\theta=\sum_{i=1}^la_i\alpha_i.$  According to \cite{L1},\cite{L2}, if $a_i=1$ then $L_{\wg}(k,k\Lambda_i)$ is a simple current and
$L_{\wg}(k,k\Lambda_i)\boxtimes L_{\wg}(k,\Lambda)=L_{\wg}(k,\Lambda^{(i)})$ for any $\Lambda\in P_+^k$
where $\Lambda^{(i)}\in P_+^k$ is uniquely determined by $\Lambda$ and $i.$
Then  $L_{\wg}(k,\Lambda)$ and $L_{\wg}(k,\Lambda^{(i)})$ are isomorphic $K(\g,k)$-modules \cite{DR}.

Here are some  main results on $K(\g,k)$ from \cite{DW2}, \cite{ALY1}, \cite{ALY2},  \cite{DR}, \cite{ADJR}.
\begin{thm}\label{DR} Let $\g$ be a simple Lie algebra and $k$ a positive integer.

(1) The $K(\g,k)$ is a rational, simple  and $C_2$-cofinite vertex operator algebra of CFT type.

(2) For any $\Lambda\in P_+^k,$ $\lambda\in \Lambda+Q$
and $\alpha\in Q_L,$ $M^{\Lambda,\lambda}=M^{\Lambda,\lambda+k\alpha}.$

(3) For each $i\in I,$ $\Lambda\in P_+^k$ there exists a unique  $\Lambda^{(i)}\in P_+^k$  such that
for any $\lambda\in \Lambda+Q,$ $M^{\Lambda,\lambda}=M^{\Lambda^{(i)},\lambda+k\Lambda_i}.$

(4) Any irreducible $K(\g,k)$-module is isomorphic to $M^{\Lambda,\lambda}$ for some $\Lambda\in P_+^k$
and $\lambda\in \Lambda+Q.$

(5) The $K(\g,k)$ has exactly $\frac{|P_+^k||Q/kQ_L|}{|P/Q|}$ inequivalent irreducible modules.
\end{thm}

\section{Trace functions for parafermion vertex operator algebras}
\setcounter{equation}{0}

In this section we determine the $S$-matrix for parafermion vertex operator algebra $K(\g,k)$ associated to any finite dimensional simple Lie algebra $\g$ and positive
integer $k$ from \cite{DR}.

From decomposition (\ref{e5.1}) we have
$$\chi_{ L_{\wg}(k,\Lambda)}(\tau)=\sum_{\lambda\in Q+\Lambda}\chi_{M_{\widehat{\h}}(k,\lambda)}(\tau)\chi_{M^{\Lambda,\lambda}}(\tau)=\sum_{\lambda\in Q+\Lambda}\frac{q^{\<\lambda,\lambda\>/2k}}{\eta(\tau)^l}\chi_{M^{\Lambda,\lambda}}(\tau)$$
where $\eta(\tau)=q^{1/24}\prod_{n\geq 1}(1-q^n).$ The
$$c^{\Lambda}_{\lambda}(\tau)=\eta(\tau)^{-l}\chi_{M^{\Lambda,\lambda}}(\tau)$$
is called the string function and $\chi_{M^{\Lambda,\lambda}}(\tau)$ is called the branching function denoted by $b_{\lambda}^{\Lambda}(\tau)$ in \cite{K}.  The modular transformation
formulas of branching functions
were also given in \cite{K}:
\begin{thm}\label{t4.2} Let $\Lambda\in P_+^k$ and $i\in Q/kQ_L.$  Then

(1)  $\chi_{M^{\Lambda,\Lambda+\beta_i}}(\tau+1)=e^{2\pi i(\frac{\<\Lambda+2\rho, \Lambda\>}{2(k+h^{\vee})}-\frac{\<\Lambda+\beta_i,\Lambda+\beta_i\>}{2k}-\frac{k\dim\g}{24(k+h^{\vee})}+\frac{l}{24})}\chi_{M^{\Lambda,\Lambda+\beta_i}}(\tau),$

(2) $\chi_{M^{\Lambda,\Lambda+\beta_i}}(\frac{-1}{\tau})=\sum_{\Lambda'\in P_+^k,j\in Q/kQ_L}S_{(\Lambda,\Lambda+\beta_i),(\Lambda',\Lambda'+\beta_i)}^P\chi_{M^{\Lambda',\Lambda'+\beta_j}}(\tau),$
where
$$S_{(\Lambda,\Lambda+\beta_i),(\Lambda',\Lambda'+\beta_i)}^P=|P/kQ_L|^{-1/2}S_{\Lambda,\Lambda'}e^{2\pi i\frac{\<\Lambda+\beta_i,\Lambda'+\beta_j\>}{k}}$$
and $P$ stands for Parafermoin.
\end{thm}

The main result in this paper is that the transformation formula for the branching functions remains valid for the trace functions $Z_{M^{\Lambda,\lambda}}(w,\tau)$ with $w\in K(\g,k):$
\begin{thm} \label{mthm}Let $\Lambda\in P_+^k$ and $i\in Q/kQ_L.$ Then for any $w\in K(\g,k)$

(1)  $Z_{M^{\Lambda,\Lambda+\beta_i}}(w,\tau+1)=e^{2\pi i(\frac{\<\Lambda+2\rho, \Lambda\>}{2(k+h^{\vee})}-\frac{\<\Lambda+\beta_i,\Lambda+\beta_i\>}{2k}-\frac{k\dim\g}{24(k+h^{\vee})}+\frac{l}{24})}Z_{M^{\Lambda,\Lambda+\beta_i}}(w,\tau),$

(2) $Z_{M^{\Lambda,\Lambda+\beta_i}}(w,\frac{-1}{\tau})=\tau^{\wt[w]}\sum_{\Lambda'\in P_+^k,j\in Q/kQ_L}S_{(\Lambda,\Lambda+\beta_i),(\Lambda',\Lambda'+\beta_i)}^PZ_{M^{\Lambda',\Lambda'+\beta_j}}(w,\tau).$
\end{thm}
\begin{proof}  (1) is straightforward.  The proof  (2) is  divided in several steps.

(a) $\chi_{ L_{\wg}(k,\Lambda)}(w,h,q)$ converges to a holomorphic function $\chi_{ L_{\wg}(k,\Lambda)}(w,h,\tau)$
for any $w\in K(\g,k)$ and $h\in \h$ where we have identified $x\in \g$ with $x(-1)\1\in L_{\wg}(k,0)_1.$

We denote the Virasoro vectors of $L_{\wg}(k,0),$ $V_{\sqrt{k}Q_L}$ and $K(\g,k)$ by
$\omega^a,$ $\omega^l$ and $\omega^p$ and denote the corresponding components of the vertex operators by
$L^a(n),$ $L^l(n)$ and $L^p(n)$  respectively  for $n\in \Z.$ Also denote the Virasoro central charges  of $L_{\wg}(k,0)$ and $K(\g,k)$ by
$c^a$ and $c^p$ respectively. Then $c^a=l+c^p.$

Recall decomposition (\ref{e5.3}).  Then
\begin{eqnarray*}
& &\chi_{ L_{\wg}(k,\Lambda)}(w,h,q)=\tr_{ L_{\wg}(k,\Lambda)}o(w)e^{2\pi i h(0)}q^{L^a(0)-c^a/24}\\
& &\ \ =\sum_{i\in Q/kQ_L}\tr_{V_{\sqrt{k}Q_L+\frac{1}{\sqrt{k}}(\Lambda+\beta_i)\otimes M^{\Lambda,\Lambda+\beta_i}}}o(w)e^{2\pi i h(0)}q^{L^a(0)-c^a/24}\\
& &\ \ =\sum_{i\in Q/kQ_L}\tr_{V_{\sqrt{k}Q_L+\frac{1}{\sqrt{k}}(\Lambda+\beta_i)}}e^{2\pi i h(0)}q^{L^l(0)-l/24}\tr_{ M^{\Lambda,\Lambda+\beta_i}}o(w)q^{L^p(0)-c^p/24}\\
& &\ \ =\sum_{i\in Q/kQ_L}\chi_{V_{\sqrt{k}Q_L+\frac{1}{\sqrt{k}}(\Lambda+\beta_i)}}(h,q)Z_{ M^{\Lambda,\Lambda+\beta_i}}(w,q).
\end{eqnarray*}
Note that $\chi_{V_{\sqrt{k}Q_L+\frac{1}{\sqrt{k}}(\Lambda+\beta_i)}}(h,q)$ converges to a holomorphic function
$\chi_{V_{\sqrt{k}Q_L+\frac{1}{\sqrt{k}}(\Lambda+\beta_i)}}(h,\tau)$ \cite{K} and $Z_{ M^{\Lambda,\Lambda+\beta_i}}(w,q)$
converges to a holomorphic function $Z_{ M^{\Lambda,\Lambda+\beta_i}}(w,\tau) $ \cite{Z}. Thus  $\chi_{ L_{\wg}(k,\Lambda)}(w,h,q)$ converges.

For short, we set $ \chi_{\Lambda}(w,h,\tau)=\chi_{ L_{\wg}(k,\Lambda)}(w,h,\tau).$

(b) By Theorems \ref{gm}, \ref{k} and Lemma \ref{same} we have
$$\chi_{\Lambda}(w,\frac{h}{\tau},\frac{-1}{\tau})=\tau^{\wt[w]}e^{\pi ik\<h,h\>/\tau} \sum_{\Lambda'\in P_+^k}S_{\Lambda,\Lambda'}\chi_{\Lambda'}(w,h,\tau).$$

(c) Let $L$ be a positive definite even lattice of rank $l$ with bilinear form $(,).$ Recall from \cite{B} and \cite{FLM} the lattice vertex operator algebra
$V_L=M(1)\otimes \C[L]$ and its irreducible modules $V_{L+\lambda_i}=M(1)\otimes \C[L+\lambda_i]$ where
$L^{\circ}=\cup_{i\in L^{\circ}/L}(L+\lambda_i).$ Then $V_L$ is a vertex operator algebra satisfying V1-V3. In this case
$$\chi_{V_{L+\lambda_i}}(h,\tau)=\tr_{V_{L+\lambda_i}}e^{2\pi i h(0)}q^{L(0)-l/24}=\frac{\theta_{L+\lambda_i}(h,\tau)}{\eta(\tau)^l}$$
where
$$\theta_{L+\lambda_i}(h,\tau)=\sum_{\lambda\in L+\lambda_i}e^{2\pi i(h,\lambda)}q^{(\lambda,\lambda)/2}.$$
It is well  known that $\{\theta_{L+\lambda}|L+\lambda\in L^{\circ}/L\}$ are linearly independent functions on $\h\times \H$
where $\h=\C\otimes_\Z L.$ Thus,  $\{\chi_{L+\lambda}|L+\lambda \in L^{\circ}/L\}$ are linearly independent functions on $\h\times \H.$
 Using the transformation formula
$$\theta_{L+\lambda}(\frac{h}{\tau},\frac{-1}{\tau})=(-i\tau)^{l/2}|L^{\circ}/L|^{-1/2}e^{\pi i(h,h)/\tau}\sum_{\lambda'+L\in L^{\circ}/L}
e^{-2\pi i(\lambda,\lambda')}\theta_{L+\lambda'}(h,\tau)$$
and
$$\eta(-\1/\tau)=(-i\tau)^{1/2}\eta(\tau)$$
we see that
$$\chi_{L+\lambda_i}(\frac{h}{\tau},\frac{-1}{\tau})=|L^{\circ}/L|^{-1/2}e^{\pi i(h,h)/\tau}\sum_{j\in L^{\circ}/L}
e^{-2\pi i(\lambda_i,\lambda_j)}\chi_{L+\lambda_j}(h,\tau).$$
Set $S_{L+\lambda,L+\lambda'}=|L^{\circ}/L|^{-1/2}e^{-2\pi i(\lambda,\lambda')}$ for $L+\lambda,L+\lambda'\in L^{\circ}/L.$
Also set  $$S_L=(S_{L+\lambda,L+\lambda'})_{L+\lambda,L+\lambda'\in L^{\circ}/L}$$
which is the $S$-matrix for the lattice vertex operator algebra $V_L.$

(d) Let  $L=\sqrt{k}Q_L.$ Then $L^{\circ}=\frac{1}{\sqrt{k}}P$  \cite{ADJR}. As in \cite{K} we consider column vector
$$\overrightarrow{\chi(w,h,\tau)}=(\chi_{\Lambda})_{\Lambda\in P_+^k},\ \ \ \ \ \overrightarrow{\chi_{\sqrt{k}Q_L}(h,\tau)}=(\chi_{L+\lambda})_{L+\lambda\in L^{\circ}/L}.$$
Let $S_A=(S_{\Lambda,\Lambda'})_{\Lambda,\Lambda'\in P_+^k}$ which is the $S$-matrix for affine vertex operator algebra $ L_{\wg}(k,0)$ (see Lemma \ref{same}).  Also consider the matrix
$$Z(w,\tau)=(Z_{M^{\Lambda, \lambda}})_{\Lambda\in P_+^k,\lambda+L\in L^{\circ}/L }$$
where $Z_{M^{\Lambda,\lambda}}=0$ if $\lambda$ does not lie in $\Lambda+Q.$ From the discussion above we see that
$$\overrightarrow{\chi(w,h,\tau)}=Z(w,\tau)\overrightarrow{\chi_{\sqrt{k}Q_L}(h,\tau)}.$$
Performing  the transformation of both sides by matrix $\left(\begin{array}{cc} 0 & -1\\ 1 & 0\end{array}\right)$  we see that
$$\tau^{\wt[w]}e^{\pi ik\<h,h\>/\tau}S_A\overrightarrow{\chi(w,h,\tau)}=Z(w,\frac{-1}{\tau})e^{\pi ik\<h,h\>/\tau}S_L\overrightarrow{\chi_{\sqrt{k}Q_L}(h,\tau)}.$$
Or equivalently,
$$\tau^{\wt[w]}S_AZ(w,\tau)\overrightarrow{\chi_{\sqrt{k}Q_L}(h,\tau)}=Z(w,\frac{-1}{\tau})S_L\overrightarrow{\chi_{\sqrt{k}Q_L}(h,\tau)}.$$
The linear independence of functions $\{\chi_{L+\lambda}|L+\lambda\in L^{\circ}/L\}$ implies that
$$\tau^{\wt[w]}S_AZ(w,\tau)=Z(w,\frac{-1}{\tau})S_L.$$
It is well  known from \cite{K} that  $S_L$ is symmetric, unitary and $S_L^{-1}=\overline{S_L}.$  In fact, these properties hold for the $S$-matrix associated to any strong rational vertex operator algebra \cite{DLN}. Finally we deduce
$$Z(w,\frac{-1}{\tau})=\tau^{\wt[w]}S_AZ(w,\tau)\overline{S_L}.$$

(e) Comparing the $(\Lambda,\lambda)$-entries of the both sides gives
\begin{eqnarray*}
& &Z_{M^{\Lambda,\lambda}}(w,\frac{-1}{\tau})=\tau^{\wt[w]}\sum_{\Lambda'\in P_+^k,\lambda'\in P/kQ_L}S_{\Lambda,\Lambda'}\overline{S_{\lambda,\lambda'}}Z_{M^{\Lambda',\lambda'}}(w,\tau)\\
& &\ \ =\tau^{\wt[w]}|P/kQ_L|^{-1/2}\sum_{\Lambda'\in P_+^k,\lambda'\in P/kQ_L}S_{\Lambda,\Lambda'}\e^{2\pi i\frac{\<\lambda,\lambda'\>}{k}}Z_{M^{\Lambda',\lambda'}}(w,\tau).
\end{eqnarray*}

Now we take $\lambda=\Lambda+\beta_i$ for $i\in Q/kQ_L.$ Note that $M^{\Lambda',\lambda'}$ is nonzero if and only if
$\lambda'+kQ_L=\Lambda'+\beta_j+kQ_L$ for some $j\in Q/kQ_L.$ As a result, we see that
$$Z_{M^{\Lambda,\Lambda+\beta_i}}(w,\frac{-1}{\tau})=\tau^{\wt[w]}\sum_{\Lambda'\in P_+^k,j\in Q/kQ_L}S_{(\Lambda,\Lambda+\beta_i),(\Lambda',\Lambda'+\beta_j)}^PZ_{M^{\Lambda',\Lambda'+\beta_j}}(w,\tau)$$
and the proof is complete.
\end{proof}


\section{Connection with orbifold theory}
\setcounter{equation}{0}

Set $H=\frac{1}{k}P.$ For $\alpha\in H$ we define  $g_{\alpha}=e^{2\pi i\alpha(0)}$ where we have identify ${\frak h}$ with ${\frak h}^*$ via the bilinear form $\<,\>.$  Then $g_{\alpha}$ acts on $ L_{\wg}(k,\Lambda)$ for any $\Lambda\in P^k_+$ such that
$$g_{\alpha}Y(u,z)g_{\alpha}^{-1}=Y(g_{\alpha}u,z)$$
 for $u\in L_{\wg}(k,0).$ In particular, $g_{\alpha}$ is an automorphism of
vertex operator algebra $L_{\wg}(k,0).$ Moreover, $g_{\alpha}=1$ on $L_{\wg}(k,0)$ if and only if  $\alpha\in Q^{\circ}.$ That is,
$G=H/Q^{\circ}$ is an automorphism group of $L_{\wg}(k,0).$ For each $\beta\in Q$ we define an irreducible character
$\mu_{\beta}$ of $G$ such that $\mu_{\beta}(g_{\alpha})=g_{\alpha}(\beta).$ Following \cite{DM1} and \cite{DLM0} we use
$L_{\wg}(k,0)^{\mu_\beta}$ to denote the subspace of $L_{\wg}(k,0)$ which is a sum of irreducible $G$-submodule with character $\mu_{\beta}.$ Recall that $Q=\cup_{i\in Q/kQ_L}(kQ_L+\beta_i).$ By Lemma \ref{QLP}, $\{\mu_{\beta_i}|i\in Q/kQ_L\}$ gives a complete list of inequivalent irreducible characters of $G.$ The following result is immediate from
 \cite{DLM0}.
\begin{lem}\label{l4.1} The $L_{\wg}(k,0)$ is a completely reducible $V_{\sqrt{k}Q_L}\otimes K(\g,k)$-module
$$L_{\wg}(k,0)=\bigoplus_{i\in Q/kQ_L}L_{\wg}(k,0)^{\mu_{\beta_i}},$$
and $L_{\wg}(k,0)^{\mu_{\beta_i}}=V_{\sqrt{k}Q_L+\frac{1}{\sqrt{k}}\beta_i}\otimes M^{0,,\beta_i}$ is an irreducible
module for $V_{\sqrt{k}Q_L}\otimes K(\g,k).$ Moreover,  if $i\ne j$ then  $L_{\wg}(k,0)^{\mu_{\beta_i}}$ and
 $L_{\wg}(k,0)^{\mu_{\beta_j}}$ are inequivalent.
\end{lem}

If $\Lambda\in P_+^k$ is not zero, $ L_{\wg}(k,\Lambda)$ is still a module for $H$ but not a module for $G$ unless $\Lambda\in Q.$ However, $\alpha\mapsto \tilde{g}_{\alpha}=g_{\alpha}e^{-2\pi i\<\alpha,\Lambda\>}$ gives a $G$-module structure on
 $ L_{\wg}(k,\Lambda).$ It is clear that
\begin{equation}\label{e7.1}
\tilde{g}_{\alpha}Y(u,z)\tilde{g}_{\alpha}^{-1}=Y(g_{\alpha}u,z)
\end{equation}
on  $ L_{\wg}(k,\Lambda).$  A generalization  of Lemma \ref{l4.1} is the following:
\begin{lem}\label{l4.2} The $L_{\wg}(k,\Lambda)$ is a completely reducible $V_{\sqrt{k}Q_L}\otimes K(\g,k)$-module
$$L_{\wg}(k,\Lambda)=\bigoplus_{i\in Q/kQ_L}L_{\wg}(k,\Lambda)^{\mu_{\beta_i}},$$
and $L_{\wg}(k,\Lambda)^{\mu_{\beta_i}}=V_{\sqrt{k}Q_L+\frac{1}{\sqrt{k}}(\Lambda+\beta_i)}\otimes M^{\Lambda,\Lambda+\beta_i}$ is an irreducible
module for $V_{\sqrt{k}Q_L}\otimes K(\g,k).$ Moreover,  if $i\ne j$ then  $L_{\wg}(k,\Lambda)^{\mu_{\beta_i}}$ and
 $L_{\wg}(k,\Lambda)^{\mu_{\beta_j}}$ are inequivalent.
\end{lem}

We can strengthen Lemma \ref{l4.2}:
\begin{prop}\label{p7.4} For $\Lambda\in P_+^k$ and $i\in Q/kQ_L,$ $L_{\wg}(k,\Lambda)^{\mu_{\beta_i}}$ are inequivalent
irreducible $V_{\sqrt{k}Q_L}\otimes K(\g,k)$-modules.
\end{prop}
\begin{proof} From equation (\ref{e7.1}) we know that $L_{\wg}(k,\Lambda)\circ g_{\alpha}$ and $L_{\wg}(k,\Lambda)$ are
isomorphic  $L_{\wg}(k,0)$-modules where $L_{\wg}(k,\Lambda)\circ g_{\alpha}=L_{\wg}(k,\Lambda)$ as vector spaces
and $Y_{L_{\wg}(k,\Lambda)\circ g_{\alpha}}(v,z)=Y_{L_{\wg}(k,\Lambda)}(g_{\alpha}v,z)$ for $v\in  L_{\wg}(k,0).$ According to a general result in orbifold theory \cite{DY}, \cite{DRX},  $L_{\wg}(k,\Lambda)^{\mu_{\beta_i}}$ are inequivalent
irreducible $V_{\sqrt{k}Q_L}\otimes K(\g,k)$-modules.
\end{proof}

We can now express  the trace functions  $Z_{M^{\Lambda.\Lambda+\beta_i}}(w,\tau)$  in terms of $\chi_\Lambda(w,\alpha,\tau)$ and the characters of irreducible modules for
lattice vertex operator algebra $V_L$ with $L=\sqrt{k}Q_L.$
For $w\in K(\g,k)$ we have
\begin{eqnarray*}
& &Z_{V_{\sqrt{k}Q_L+\frac{1}{\sqrt{k}}(\Lambda+\beta_i)}\otimes M^{\Lambda,\Lambda+\beta_i}}(w,\tau)=\frac{1}{|Q/kQ_L|}\sum_{\alpha\in G}\tr_{L_{\wg}(k,\Lambda)}o(w)\tilde{g}_{\alpha}q^{L(0)-c/24}e^{-2\pi i\<\alpha,\beta_i\>}\\
& &\ \ \ =\frac{1}{|Q/kQ_L|}\sum_{\alpha\in G}\tr_{L_{\wg}(k,\Lambda)}o(w)e^{2\pi i\alpha(0)}q^{L(0)-c/24}e^{-2\pi i\<\alpha,\beta_i+\Lambda\>}\\
& &\ \  \ =\frac{1}{|Q/kQ_L|}\sum_{\alpha\in G}e^{-2\pi i\<\alpha,\beta_i+\Lambda\>}\chi_{\Lambda}(w,\alpha,\tau).
\end{eqnarray*}
This implies the following:
\begin{prop} For $\Lambda\in P_+^k,$  $i\in Q/kQ_L$ and $w\in K(\g,k)$ we have
 $$Z_{M^{\Lambda,\Lambda+\beta_i}}(w,\tau)=
 \frac{1}{|Q/kQ_L|}\frac{\eta(\tau)^l}{\theta_{\sqrt{k}Q_L+\frac{1}{\sqrt{k}}(\Lambda+\beta_i)}(\tau)}\sum_{\alpha\in G}e^{-2\pi i\<\alpha,\beta_i+\Lambda\>}
 \chi_{\Lambda}(w,\alpha,\tau).$$
\end{prop}

\end{document}